\documentclass[12pt, a4paper]{article}
\usepackage{hyperref}
\usepackage{amsmath,amsfonts,amsthm,dsfont}
\usepackage[utf8]{inputenc}
\usepackage[T1]{fontenc}
\usepackage{graphicx}
\usepackage{amssymb}
\usepackage{stmaryrd}
\usepackage{ mathrsfs }

\theoremstyle{plain}
\newtheorem{theorem}{Theorem}[section]
\newtheorem{lemma}[theorem]{Lemma}

\newtheorem{proposition}[theorem]{Proposition}

\theoremstyle{definition}

\newtheorem{remark}[theorem]{Remark}

\theoremstyle{remark}

\begin{document}

\title{Extreme quantile regression in a proportional tail framework.}
\author{Benjamin Bobbia$^*$, Clément Dombry$^*$, Davit Varron\footnote{Universit\'{e} de Bourgogne-Franche-Comt\'{e}, laboratoire de Math\'{e}matiques de Besaçon, UMR 6623}}
\maketitle
\date{}

\begin{abstract}
We revisit the model of heteroscedastic extremes initially introduced by Einmahl et al. (JRSSB, 2016) to describe the evolution of a non stationary sequence whose extremes evolve over time and adapt it into a general extreme quantile regression framework. We provide estimates for the extreme value index and the integrated skedasis function and prove their asymptotic normality. Our results are quite similar to those developed for heteroscedastic extremes but with a different proof approach  emphasizing coupling arguments. We also propose a pointwise estimator of the skedasis function and a Weissman estimator of the conditional extreme quantile  and prove the asymptotic normality of both estimators.
\end{abstract}
\section{Introduction and main results}
\subsection{Framework}
One of the main goals of extreme value theory is to propose estimators of extreme quantiles.  Given an i.i.d. sample $Y_1,\ldots,Y_n$ with distribution $F$, one wants to estimate the quantile of order $1-\alpha_n$ defined as
 \[
 q(\alpha_n)=F^\leftarrow(1-\alpha_n)
 \] 
 with $\alpha_n\to 0$ as $n\to\infty$ and $F^\leftarrow(u)=\sup\{x\in\mathbb{R}: F(x)\geq u\}$, $u\in(0,1)$, the quantile function.  The extreme regime corresponds to the case when $\alpha_n< 1/n$  in which case extrapolation beyond the sample maximum is needed. Considering an application in hydrology, these mathematical problems corresponds to the following situations: given a  record over $n=50$ years of the level of a river, can we estimate the $100$-year return level ? The answer to this question is provided by the univariate extreme value theory and we refer to the monographs by Coles \cite{C01}, Beirlant et al. \cite{BGTS04} or de Haan and Ferreira \cite{dehaanferreira} for a general background.
 
 In many situations, auxiliary information is available and represented by a covariate $X$ taking value in $\mathbb{R}^d$ and one want to estimate $q(\alpha_n| x)$, the conditional $(1-\alpha_n)$-quantile of $Y$  with respect to some given values of the covariate $X=x$. This yields the extreme quantile regression problem. Recent advances in extreme quantile regression include the works by Chernozhukov \cite{C05}, El Methni et al. \cite{EMGGG12} or Daouia et al. \cite{DGG13}. 
 
In this paper we develop the proportional tail framework for extreme quantile regression that is an adaptation of the  heteroscedastic extreme framework developed by Einmahl et al.  \cite{EdHZ16}, where the authors propose a model for the extremes of independent but non stationary observations whose distribution evolves over time. The model can be viewed as a regression framework with  time as covariate and  deterministic design with uniformly distributed observation times $1/n,2/n,\ldots, 1$. In our setting, the covariate $X$ takes value in $\mathbb{R}^d$ and is random with arbitrary distribution. The main assumption, directly adapted from Einmmahl et al. \cite{EdHZ16}, is the so called proportional tail assumption formulated in Equation~\eqref{skedasis} and stating that the conditional tail function of $Y$ given $X=x$ is asymptotically proportional to the unconditional tail. The proportionality factor is given by the so called skedasis function $\sigma(x)$ that accounts for the dependency of the extremes of $Y$ with respect to the covariate $X$. Furthermore, as it is standard in extreme value theory, the unconditional distribution of $Y$ is assumed to be regularly varying. Together with the proportional tail assumption, this implies that all the conditional distributions are regularly varying with the same extreme value index. Hence the proportional tail framework  appears suitable for modeling covariate dependent extremes where the  extreme value index is constant but  the scale parameter depends on the covariate $X$ in a non parametric way related to the skedasis function $\sigma(x)$. Note that this framework is also considered by Gardes \cite{G15} for the purpose of estimation of the extreme value index.

Our main results are presented in the following subsections. Section~\ref{model presentation} considers the estimation of the extreme value index and  integrated skedasis function in the proportional tail model and our results of asymptotic normality are  similar to those in Einmahl et al. \cite{DEdH89} but with a different proof  emphasizing coupling arguments. Section~\ref{sec:extreme-quantile} considers pointwise estimation of the skedasis function and  conditional extreme quantile estimation with Weissman estimators and state their asymptotic normality. Section~\ref{sec:coupling} develops some coupling arguments used in the proofs of the main theorems, proofs  gathered in Section~\ref{sec:proofs}. Finally, an appendix states a technical lemma and its proof. 

\subsection{The proportional tail model}\label{model presentation}

Let $(X,Y)$ be a generic random couple taking values in $\mathbb{R}^d\times \mathbb{R}$. Define the conditional cumulative distribution function of $Y$ given $X=x$ by
$$
F_x(y)=\mathbb{P}(Y \leq y |X=x), \ y \in \mathbb{R}, \ x \in \mathbb{R}^d.
$$
The main assumption of \textit{the proportional tail model} is
\begin{equation}\label{skedasis}
\underset{y \rightarrow +\infty}{\lim}\frac{1-F_x(y)}{1-F^0(y)}=\sigma(x) \quad \text{uniformly in } x \in \mathbb{R}^d,
\end{equation}
where $F^0$ is some baseline distribution function with upper endpoint $y^*$ and $\sigma$ is the so-called \underline{skedasis function} following the terminology introduced in \cite{EdHZ16}.
By integration, the unconditional distribution $F$ of $Y$ satisfies 
$$
\underset{y \rightarrow +\infty}{\lim}\frac{1-F(y)}{1-F^0(y)}=\int_{\mathbb{R}^d}\sigma(x) \mathbf{P}_X(dx).
$$
We can hence suppose without loss of generality that $F=F^0$ and $\sigma$ is a $ \mathbf{P}_X$-density function.

\medskip
For the purpose of estimation of the tail distribution in the regime of extreme, we assume furthermore that $F$ belongs to the domain of attraction of some extreme-value distribution $G_\gamma$ with $\gamma >0$,  denoted by $F \in D(G_\gamma)$. Equivalently,  $F$ is of Pareto-type with $\alpha=1/\gamma$. Together with the proportional tail condition  $\eqref{skedasis}$ with $F=F^0$, this implies that $F_x \in D(G_\gamma)$ for all $x\in\mathbb{R}^d$ with constant extreme value index. This is a strong consequence of the model assumptions.

In this model, the extreme are driven  by two parameters: the extreme value index $\gamma>0$ and the skedasis function $\sigma$. Following \cite{EdHZ16}, we consider the Hill estimator for $\gamma$ and a non-parametric estimator of the integrated  skedasis function
$$
C(x)=\int_{\{ u \leq x \}} \sigma(u) \mathbf{P}_X(du),\quad x\in\mathbb{R}^d,
$$
where  $u \leq x$ stands for the componentwise comparison of vectors. Note that $C$ is easier to estimate than $\sigma$, in the same way that a cumulative distribution function is easier to estimate than a density function. Estimation of $C$ is useful to derive tests while estimation of $\sigma$ will be considered later on for the purpose of extreme quantile estimation.

Let $(X_i,Y_i)_{1\leq i \leq n}$ be i.i.d copies of $(X,Y)$. The estimators are built with observations $(X_i,Y_i)$ for which $Y_i$ exceeds a high threshold $\textbf{y}_n$. The real valued threshold $(\textbf{y}_n)_{n \in \mathbb{N}}$ can be deterministic or data driven. For the purpose of asymptotics, $\textbf{y}_n$ depends on the sample size $n\geq 1$ in a way such 
\[
\textbf{y}_n\to\infty\quad \mbox{and}\quad N_n\to\infty
\quad \mbox{in probability}
\] 
with $N_n:=\sum_{i=1}^n\mathds{1}_{\{Y_i > \textbf{y}_n\}}$ the (possibly random) number of exceedances.
The extreme value index $\gamma>0$ is estimated by the Hill-type estimator 
$$
\hat{\gamma}_n=\frac{1}{N_n}\underset{i=1}{\overset{n}{\sum}}\log\left(\frac{Y_{i}}{\mathbf{y}_{n}}\right) \mathds{1}_{\left\{Y_i>\mathbf{y}_{n}\right\}}.
$$
The integrated skedasis function $C$ can be estimated by the pseudo empirical distribution function
$$
\widehat{C}_n(x):=\frac{1}{N_n}
\sum_{i=1}^{n}  \mathds{1}_{\left\{Y_i>\mathbf{y}_{n} ,\ X_i\leq x\right\}}, \ x \in \mathbb{R}^d.
$$
When $Y$ is continuous and $\textbf{y}_n:=Y_{n-k_n;n}$ is the $(k_n+1)$-th top order statistic, then $N_n=k$ and $\hat{\gamma}_n$ is the usual Hill estimator.

Our first result addresses the joint asymptotic normality of $\hat\gamma_n$ and $\widehat{C}_n$, namely
\begin{equation}\label{EdHZ conv}
v_n\begin{pmatrix}
\widehat{C}_n(\cdot)-C(\cdot) \\
\hat{\gamma}_n-\gamma
\end{pmatrix}\underset{n \rightarrow +\infty}{\overset{\mathscr{L}}{\longrightarrow}}\mathbb{W},
\end{equation}
where $\mathbb{W}$ is a Gaussian Borel probability measure on $L^\infty(\mathbb{R}^d)\times \mathbb{R}$, and $v_n \rightarrow +\infty$ is a deterministic rate. To prove the asymptotic normality, the threshold $\mathbf{y}_n$ must scale suitably with respect to the rates of convergence in the proportional tail and domain of attraction conditions. More precisely, we assume the existence of a positive function $A$ converging to zero and such that, as $y\to\infty$,
\begin{equation}\label{second order 1}
\underset{x \in \mathbb{R}^d}{\sup}\left|\frac{\bar{F}_x(y)}{\sigma(x)\bar{F}(y)}-1\right|=\mathrm{O}\left(A\left(\frac{1}{\bar{F}(y)}\right)\right)
\end{equation}
and
\begin{equation}\label{second order 2}
 \underset{z >\frac{1}{2}}{\sup}\left|\frac{\bar{F}(zy)}{z^{-\alpha}\bar{F}(y)}-1\right|=\mathrm{O}\left(A\left(\frac{1}{\bar{F}(y)}\right)\right)
\end{equation}
with $\bar{F}(y):=1-F(y)$ and $\bar F_x(y):=1-F_x(y)$. Our main result can then be stated as follows.

\begin{theorem}\label{EdHZ}
Assume that assumptions \eqref{second order 1} and \eqref{second order 2} hold and that
$\mathbf{y}_n/y_n\to 1$ in probability for some deterministic sequence $y_n$ such that $p_n:=\bar{F}(y_n)$ satisfies
$$ p_n\to 0,\quad  np_n \rightarrow +\infty\quad \mbox{and} \quad \sqrt{np_n}^{1+\varepsilon}A\left(\frac{1}{p_n}\right)\rightarrow 0 \ \text{for some } \varepsilon >0.$$
	Then,  the asymptotic normality \eqref{EdHZ conv} holds with 
	$$
	v_n:=\sqrt{np_n} \quad \text{and} \quad
	\mathbb{W}\overset{\mathscr{L}}{=}\begin{pmatrix}
	B \\ N
	\end{pmatrix},
	$$
	with $B$ a $C$-Brownian bridge on $\mathbb{R}^d$ and $N$ a centered Gaussian random variable with variance $\gamma^2$ and independent of $B$.
\end{theorem}

\begin{remark}
While Theorem \ref{EdHZ} extends Theorem 2.1 in Einmhal \textit{et al}. \cite{EdHZ16}, that extension comes at the price of a slightly more stringent condition upon the bias control. Indeed, their condition $\sqrt{k_n}A(n/k_n)\rightarrow 0$ corresponds to our condition $\sqrt{np_n}^{1+\varepsilon}A(1/p_n)\rightarrow 0$ with $\varepsilon=0$. We believe that this loss is small in regard to the more general choice of $\textbf{y}_n$. For example, the data-driven threshold $\mathbf{y}_n:=Y_{n-k_n;n}$ is equivalent in probability to $y_n:=F^\leftarrow\left(1-\frac{k_n}{n}\right)$, with $F^\leftarrow(p)=\inf\{y | F(y)\geq p\}$, $p\in (0,1)$ the pseudo-inverse of $F$. As a consequence, Theorem \ref{EdHZ} holds for this choice of $\textbf{y}_n$ if 
\[
k_n\rightarrow +\infty, \ \frac{k_n}{n} \rightarrow 0, \ \text{and} \ \sqrt{k_n}^{1+\varepsilon}A\left(\frac{n}{k_n}\right)\rightarrow 0.
\]
\end{remark}

\subsection{Extreme quantile regression}\label{sec:extreme-quantile}
We now adress the estimation of extreme conditional quantiles in the proportional tail model. The conditional quantile of order $1-\alpha$ given $X=x$ is defined by 
\[
q(\alpha|x)=F_x^\leftarrow (1-\alpha), \quad \alpha\in (0,1).
\]
We say that we estimate an extreme quantile when $\alpha$ is small with respect to the inverse sample size $1/n$; more formally, we will here consider the asymptotic regime $\alpha=\alpha_n$ with $\alpha_n=O(1/n)$ as $n\to\infty$.

Let us first remind the reader about some notions of  extreme quantile estimation in a non conditional framework and the heursitic idea behind the Weissman  estimator \cite{W78}. For quantile estimation, it is natural to consider $\mathbb{F}_n^\leftarrow(\alpha)$, the inverse of the empirical distribution function $\mathbb{F}_n$. However, in the extreme regime $\alpha<1/n$, this estimator is nothing but the maximum of the sample and we see that some kind of extrapolation beyond the sample maximum is needed. Such an extrapolation relies on the first order condition $F\in D(G_\gamma)$ which is equivalent to the regular variation of the tail quantile function 
\[
\underset{t \rightarrow \infty}{\lim}\frac{U(tz)}{U(t)}=z^\gamma, \quad z>0,
\] 
with $U(t)=F^\leftarrow (1-1/t)$, $t>1$. Consider a high threshold $\textbf{y}_n$ and write  $p_n=\bar{F}(\textbf{y}_n)$. Using regular variation, the quantile $q(\alpha_n):=F^\leftarrow(1-\alpha_n)$ is approximated by 
\[
q(\alpha_n)=U(1/p_n)\frac{U(1/\alpha_n)}{U(1/p_n)}\approx\textbf{y}_n\left(\frac{p_n}{\alpha_n}\right)^\gamma,
\]
leading to the Weissman-type quantile estimator
\[
\hat{q}(\alpha_n):=\textbf{y}_{n}\left(\frac{\hat{p}_n}{\alpha_n}\right)^{\hat{\gamma}_n}.
\]
Where $\hat{p}_n:=\frac{1}{n}\sum_{i=1}^n\mathds{1}_{\{Y_i> \mathbf{y}_n\}}$ is an estimator of $p_n=\mathbb{P}(Y>\textbf{y}_n)$ depending on the unconditional distribution of $Y$.   

Now going back to quantile regression in the proportional tail model, it is readily verified that assumption \eqref{skedasis} implies 
\[
q(\alpha_n\mid x)\sim q\left(\frac{\alpha_n}{\sigma(x)}\right)\quad \mbox{as $n\to\infty$}.
\]
This leads to the plug-in estimator 
\[
\hat{q}(\alpha_n|x):=\hat q\left(\frac{\alpha_n}{\hat \sigma_n(x)}\right)=\mathbf{y}_n \left(\frac{\hat{p}_n\hat{\sigma}_n(x)}{\alpha_n}\right)^{\hat{\gamma}_n}
\]
where $\hat\sigma_n(x)$ denotes a consistent estimator of $\sigma(x)$.

In the following, we  propose a kernel estimator of $\sigma(x)$ and prove its asymptotic normality before deriving the asymptotic normality of the extreme conditional quantile estimator $\hat{q}(\alpha_n|x)$.  The proportional tail assumption \eqref{skedasis} implies
\[
\sigma(x)=\underset{n \rightarrow +\infty}{\lim}\frac{\mathbb{P}(Y> \mathbf{y}_n|X=x)}{\mathbb{P}(Y > \mathbf{y}_n)}.
\]
We propose the simple kernel estimator with bandwidth $h_n>0$
\[
\mathbb{P}(Y> \mathbf{y_n}|X=x)\approx \frac{\sum_{i=1}^n\mathds{1}_{\{|x-X_i|< h_n\}}\mathds{1}_{\{Y_i> \mathbf{y}_n\}}}{\sum_{i=1}^n\mathds{1}_{\{|x-X_i|< h_n\}}}
\]
to estimate the numerator, while the denominator is estimated  empirically by $\hat{p}_n$.
Combining the two estimates, the skedasis function at $x$ is estimated by
\[
\hat\sigma_n(x)=n\frac{\sum_{i=1}^n\mathds{1}_{\{|x-X_i|< h_n\}}\mathds{1}_{\{Y_i> \mathbf{y}_n\}}}{\sum_{i=1}^n\mathds{1}_{\{|x-X_i|< h_n\}}\sum_{i=1}^n\mathds{1}_{\{Y_i> \mathbf{y}_n\}}}.
\]

\begin{theorem}\label{thm:normality-skedasis}
Let $\mathbf{y}_n\to\infty $ be a deterministic sequence and set $p_n:=\bar{F}(\mathbf{y}_n)$. Let $h_n\to 0$ and assume that
\[
n p_n h_n^d \rightarrow +\infty,\quad \sqrt{np_nh_n^d}A\left(\frac{1}{p_n}\right)\rightarrow 0.
\] 
Assume that both $f$ and $\sigma$ are continuous and positive on a neighborhood of $x\in\mathbb{R}^d$. Then, under assumption \eqref{second order 1}, we have
\[
\sqrt{np_nh_n^d}\Big(\hat{\sigma}_n(x)-\sigma(x)\Big)\overset{\mathscr{L}}{\underset{n \rightarrow +\infty}{\longrightarrow}} \mathcal{N}\left(0,\frac{\sigma(x)}{f(x)}\right).
\]
\end{theorem}
The asymptotic normality of the extreme quantile estimate $\hat q(\alpha_n\mid x)$ is deduced from the asymptotic normality of $\hat\gamma_n$ and $\hat\sigma_n(x)$ stated respectively in Theorem \ref{EdHZ} and \ref{thm:normality-skedasis}.
\begin{theorem}\label{thm:quantile-normality}
Under assumptions of Theorems \ref{EdHZ} and \ref{thm:normality-skedasis}, if $\sqrt{h_n^d}\log(p_n/\alpha_n)\rightarrow +\infty$ we have
\[
\frac{\sqrt{np_n}}{\log(p_n/\alpha_n)}\log\Big(\frac{\hat{q}(\alpha_n|x)}{q(\alpha_n|x)}\Big)\overset{\mathscr{L}}{\underset{n \rightarrow +\infty}{\longrightarrow}} \mathcal{N}\left(0,\gamma^2\right).
\]
\end{theorem}
The condition $\sqrt{h_n^d}\log(p_n/\alpha_n)$ requires the bandwidth to be of larger order than $1/\log(p_n/\alpha_n)$ so that the error in the estimation of $\sigma(x)$ is negligible. We restrict ourselves here to the case of a deterministic threshold $\textbf{y}_n$, the general case  of a data driven threshold is left for future work.

As a consequence of Theorem~\ref{thm:quantile-normality}, the consistency  
\[
\frac{\hat{q}(\alpha_n|x)}{q(\alpha_n|x)}\overset{\mathbb{P}}{\rightarrow}1
\]
holds if and only if  $\log(n\alpha_n)=o(\sqrt{np_n})$. Together with the delta-method, this also implies the more classical asymptotic normality 
\[
\frac{\sqrt{np_n}}{\log(p_n/\alpha_n)}\Big(\frac{\hat{q}(\alpha_n|x)}{q(\alpha_n|x)}-1\Big)\overset{\mathscr{L}}{\underset{n \rightarrow +\infty}{\longrightarrow}} \mathcal{N}\left(0,\gamma^2\right).
\]
The condition $\log(n\alpha_n)=o(\sqrt{np_n})$ provides a limit for the extrapolation since $\alpha_n$ cannot be too small or one might lose consistency. 

\section{A coupling approach}\label{sec:coupling}

We will first prove Theorem \ref{EdHZ} when $\mathbf{y}_n$ is deterministic (i.e $\mathbf{y}_n\equiv y_n)$. In this case $N_n$ is binomial $(n,p_n)$. Moreover $N_n/np_n \rightarrow 1$ in probability since $np_n \rightarrow +\infty$. \\
A simple calculus shows that (\ref{skedasis}) entails, for each $A$ Borel and $t \geq 1$:
\begin{equation}\label{coupling heuristic}
\mathbb{P}\left(\frac{Y}{\mathbf{y}}\geq t, X \in A  \middle| Y \geq \mathbf{y}\right)\underset{\mathbf{y}\rightarrow \infty}{\longrightarrow}\int_t^\infty\int_A\mathbf{y}^{-\alpha}\sigma(x)\mathrm{d}\mathbf{y}\mathbf{P}_X(\mathrm{d}x),
\end{equation}
defining a "limit model" for $(X,Y/\mathbf{y})$, the law 
\[
Q:=\sigma(x)\mathbf{P}_X\otimes Pareto(\alpha)
\]
of a generic random vector $(U,V)$ with independent marginals.
Using the heuristic of (\ref{coupling heuristic}), we shall now build a coupling $(X_i,Y_i,X_{i,n}^*,Y_{i,n}^*)_{1\leq i\leq n}$ between the exceedances of our model above $y_n$ and the limit model. Define the conditional tail quantiles function as $U_{x}(t):=F_{x}^\leftarrow(1-1/t)$ and recall that the total variation norm between two probability measures on $\mathbb{R}^d$ is defined as
$$
||P_1-P_2||_{TV}:=\underset{B \text{ Borel}}{\sup}|P_1(B)-P_2(B)|.
$$

This distance is closely related with the notion of optimal coupling detailed in \cite{lindvall92}. 
\begin{lemma}[Optimal coupling]\label{maximal coupling lemma}
For two probability measures $P_1$ and $P_2$ defined on a measurable space $(E,\mathcal{E})$ there exist two random variables from a probability set $(\Omega, \mathcal{A}, \mathbb{P})$ to $(E,\mathcal{E})$ such that
	
\[
		X_1 \sim P_1, \quad X_2 \sim P_2 \quad\text{and} \quad ||P_1-P_2||_{TV}= \mathbb{P}(X_1\neq X_2). 
\]
\end{lemma}

This maximal coupling is a crucial tool of our coupling construction, which is described by the following algorithm.  

\medskip
\textbf{Coupling construction:}
Fix $n \geq 1$. Let $(E_{i,n})_{1\leq i\leq n}$ be    i.i.d. Bernoulli random variables with $\mathbb{P}(E_{i,n}=1)=\bar{F}(y_n)$ and $(Z_i)_{1\leq i\leq n}$ i.i.d. with distribution $\mathrm{Pareto}(1)$ and independent from $(E_{i,n})_{1\leq i\leq n}$.\\
For each $1\leq i \leq n$ build $(\tilde{X}_{i,n},\tilde{Y}_{i,n},X_{i,n}^*,Y_{i,n}^*)$ as follows.

\begin{itemize}
	\item[$\blacktriangleright$] If $E_{i,n}=1$, then
		\begin{itemize}
			\item[$\triangleright$] Randomly generate $\tilde{X}_{i,n} \sim P_{X\mid Y>y_n}$,\ $X_{i,n}^* \sim \sigma(x)\mathbf{P}_X(dx)$, satisfying $\mathbb{P}(\tilde{X}_{i,n} \neq X_{i,n}^*) =\|\mathbf{P}_{X|Y>y_n}-\sigma(x)\mathbf{P}_X(dx)\|_{TV}$.
			\item[$\triangleright$] Set $\tilde{Y}_{i,n}:= U_{\tilde{X}_{i,n}}(\frac{Z_i}{\bar{F}_{X_{i,n}}(y_n)})$, $Y_{i,n}^*:=y_n Z_i^{\frac{1}{\alpha}}$.
		\end{itemize}
	
	\item[$\blacktriangleright$] If $E_{i,n}=0$, then 
		\begin{itemize}
			\item[$\triangleright$] Randomly generate $(\tilde{X}_{i,n},\tilde{Y}_{i,n})\sim \mathbf{P}_{(X,Y)|Y\leq y_n}$.
			\item[$\triangleright$] Randomly generate $(X_{i,n}^*,Y_{i,n}^*/y_n)\sim \sigma(x)\mathbf{P}_X(\mathrm{d}x)\otimes Pareto(\alpha)$.
		\end{itemize}
\end{itemize}
The following proposition states the properties of our coupling construction, which will play an essential role in our proof of Theorem \ref{EdHZ}.
\begin{proposition}\label{coupling propertie}
For each $n \geq 1$, the coupling $(\tilde{X}_{i,n},\tilde{Y}_{i,n},X_{i,n}^*,Y_{i,n}^*)_{1 \leq i \leq n}$ generated by the preceding algorithm has the following properties:
\begin{enumerate}
\item $(\tilde{X}_{i,n},\tilde{Y}_{i,n})_{1 \leq i \leq n}$ has the same law as $(X_i,Y_i)_{1\leq i \leq n}$.
\item $\mathcal{L}(X_{i,n}^*,Y_{i,n}^*/y_n)=\sigma(x)\mathbf{P}_X(\mathrm{d}x)\otimes Pareto(\alpha)$.
\item $(X_{i,n}^*,Y_{i,n}^*)_{1\leq i \leq n}$ and $(E_{i,n})_{1 \leq i\leq n}$ are independent. Moreover, $(Y_{i,n}^*)_{1\leq i \leq n}$ are i.i.d, and independent from $(\tilde{X}_{i,n},X_{i,n}^*)$.
\item There exists $M >0$ such that 
	    \begin{equation}\label{Y_Y*_bound}
		\underset{\overset{1 \leq i \leq n,}{ E_{i;n=1}} }{\max}\left|\frac{Y_{i,n}^*}{\tilde{Y}_{i,n}}-1\right|\leq M A\left(\frac{1}{\bar{F}(y_n)}\right) \quad \text{and}
		\end{equation}
		\begin{equation}\label{X_X*_bound}
		\mathbb{P}\left(\tilde{X}_{1,n}\neq X_{1,n}^*|E_{i,n}=1\right)\leq MA\left(\frac{1}{\bar{F}(y_n)}\right),
		\end{equation}
		where $A$ is given by assumptions \eqref{second order 1} and \eqref{second order 2}.
\end{enumerate}

\end{proposition}

\begin{proof}
To prove Point 1, it is sufficient to see that 
\[
\mathscr{L}((\tilde{X}_{1,n},\tilde{Y}_{1,n})|E_{i,n}=1)=\mathscr{L}((X_1,Y_1)|Y_1>y_n).
\]
Since $U_x(z/(1-F_x(y_n)))\leq y$ if and only if $1-(1-F_x(y_n))/z \leq F_x(y) $ we have, for $y \geq y_n$:
\begin{align*}
&\int_1^{+\infty}\mathds{1}_{\{U_x(z/(1-F_x(y_n)))\leq y\}}\frac{\mathrm{d}z}{z^2} \\
=& \int_1^{+\infty}\mathds{1}_{\{1-(1-F_x(y_n))/z \leq F_x(y)\}}\frac{\mathrm{d}z}{z^2}\\
                       =&\int_{F_x(y_n)}^1\mathds{1}_{\{t \leq F_x(y)\}}\frac{\mathrm{d}t}{1-F_x(y_n)} \\
                       =&\int_{F_x(y_n)}^{F_x(y)}\frac{\mathrm{d}t}{1-F_x(y_n)}=\frac{F_x(y)-F_x(y_n)}{1-F_x(y_n)},
\end{align*}
with the second equality given by the change of variable $t=1-(1-F_x(y_n))/z$. We can deduce from this computation that, for a Borel set $B$ and $y \geq y_n$,
\begin{align*}
&\mathbb{P}\left(\tilde{X}_{1,n}\in B, U_{\tilde{X}_{1,n}}\left(\frac{Z}{1-F_{\tilde{X}}(y_n)}\right)\leq y\middle|E_{1,n}=1\right) \\
=&\int_{x \in B}\int_1^{+\infty}\mathds{1}_{\{U_x(z/(1-F_x(y_n)))\leq y\}}\frac{\mathrm{d}z}{z^2}\mathrm{d}P_{X|Y>y_n}(x) \\
=&\int_{x \in B}\frac{F_x(y)-F_x(y_n)}{1-F_x(y_n)}\mathrm{d}P_{X|Y>y_n}(x) \\
=&\int_{x \in B}\frac{\mathbb{P}\big(y_n<Y\leq y| X=x\big)}{\mathbb{P}\big(Y>y_n|X=x\big)}\mathrm{d}P_{X|Y>y_n}(x) \\
=&\int_{x \in B}\mathbb{P}\big(Y\leq y|Y >y_n, X=x\big)\mathrm{d}P_{X|Y>y_n}(x) \\
=&\mathbb{P}\big(X \in B, Y \leq y | Y > y_n\big).
\end{align*}

\medskip

This proves Point 1. Point 2 is proved by noticing that $Z_i\sim \mathrm{Pareto}(1)$. Consequently, we have $\mathbb{P}(Z_i >1)=z^{-1}$ for $z>1$. Hence
$$
\mathbb{P}\left(y_nZ_i^{1/\alpha}>z\right)=\mathbb{P}\left(Z_i>\left(\frac{z}{y_n}\right)^\alpha\right)=\left(\frac{y_n}{z}\right)^{-\alpha}.
$$

Hence, by construction, $Y_{1,n}^*\sim \mathrm{Pareto}(\alpha, y_n)$.
Point 3 is obvious. \\
Point 4 will be proved with the two following lemmas.

\begin{lemma}\label{lemme couplage}
Under conditions (\ref{second order 1}) and (\ref{second order 2}), we have
$$ \underset{z \geq 1}{\sup}\ \underset{x \in \mathbb{R}^p}{\sup}\left|\frac{U_x(\frac{z}{\bar{F}_x(y)})}{z^\gamma y}-1\right|=\mathrm{O}\left(A\left(\frac{1}{\bar{F}(y)}\right)\right).$$

\end{lemma}

\begin{proof}

According to assumptions (\ref{second order 1}) and (\ref{second order 2}), there exists a constant $M_0$ such that 

\begin{equation*}
\left|\frac{\bar{F}_x(y)}{\sigma(x)\bar{F}(y)}-1\right|\leq M_0 A\left(\frac{1}{\bar{F}(y)}\right), \ \text{uniformly in} \ x \in \mathbb{R}^d, \text{ and} 
\end{equation*}

\begin{equation*}
\left|\frac{\bar{F}(zy)}{z^{-\alpha}\bar{F}(y)}-1\right|\leq M_0 A\left(\frac{1}{\bar{F}(y)}\right), \ \text{uniformly in} \ z \geq 1/2 \text{, as } y \rightarrow +\infty.
\end{equation*}
From the definition of $U_x$ we have 
$$\begin{array}{ccl}
 U_x(\frac{Z}{\bar{F}_x(y)})
           &=& F_x^\leftarrow \left(1-\frac{\bar{F}_x(y)}{z}\right) \\
					 &=& \inf\left\{w \in \mathbb{R}:F_x(w) \geq 1-\frac{\bar{F}_x(y)}{z}\right\} \\
					 &=& \inf\left\{w \in \mathbb{R}: z \frac{\bar{F}_x(w)}{\bar{F}_x(y)} \leq 1\right\}.
\end{array}$$
Hence
\begin{equation}\label{omega encadrement}
z\frac{\bar{F}_x(w^+)}{\bar{F}_x(y)}<1<z\frac{\bar{F}_x(w^-)}{\bar{F}_x(y)}\Rightarrow U_x\left(\frac{z}{\bar{F}_x(y)}\right) \in \left[w^-;w^+\right].
\end{equation}
Now write $w^{\pm}:=z^\gamma y \left(1\pm 4 \gamma M_0A\left(\frac{1}{\bar{F}(y)}\right)\right)$, so that one can write
$$\begin{array}{ccl}
z\frac{\bar{F}_x(w^-)}{\bar{F}_x(y)}&=&z\frac{\sigma(x)\bar{F}(\omega^-)(1-M_0A(\frac{1}{\bar{F}(y)}))}{\sigma(x)\bar{F}(y)(1+M_0A(\frac{1}{\bar{F}(y)}))} \\
                                    &\geq& z\frac{1-M_0A(\frac{1}{\bar{F}(y)})}{1+M_0A(\frac{1}{\bar{F}(y)})}\frac{1}{\bar{F}(y)}\bar{F}\left(z^\gamma y (1- 4 \gamma M_0A(\frac{1}{\bar{F}(y)}))\right) \\
																		&\geq& z\frac{1-M_0A(\frac{1}{\bar{F}(y)})}{1+M_0A(\frac{1}{\bar{F}(y)})}\frac{1}{\bar{F}(y)}\bar{F}(y)(1-M_0A(\frac{1}{\bar{F}(y)})\left(z^\gamma(1-4 \gamma M_0A(\frac{1}{\bar{F}(y)}))\right)^{-\alpha}  \\
																		&\geq& \frac{(1-M_0A(\frac{1}{\bar{F}(y)}))^2}{1+M_0A(\frac{1}{\bar{F}(y)})}\left(1-4 \gamma M_0A(\frac{1}{\bar{F}(y)})\right)^{-\alpha}. \\
			
\end{array}$$
A similar computation gives 

$$
z\frac{\bar{F}_x(w^+)}{\bar{F}_x(y)} \leq \frac{(1+M_0A(\frac{1}{\bar{F}(y)}))^2}{1-M_0A(\frac{1}{\bar{F}(y)})}\left(1+4 \gamma M_0A(\frac{1}{\bar{F}(y)})\right)^{-\alpha}.
$$
As a consequence the condition before "$\Rightarrow$" in (\ref{omega encadrement}) holds if 
$$
4 \gamma M_0 \geq \frac{1}{A(\frac{1}{\bar{F}(y)})}\max\left\lbrace 1-\left(\frac{(1-M_0A(\frac{1}{\bar{F}(y)}))^2}{1+M_0A(\frac{1}{\bar{F}(y)})}\right)^\gamma; \left(\frac{(1+M_0A(\frac{1}{\bar{F}(y)}))^2}{1-M_0A(\frac{1}{\bar{F}(y)})}\right)^\gamma -1\right\rbrace.
$$
But a Taylor expansion of the right hand side shows that it is $3\gamma M_0  + o(1) \text{ as }y \rightarrow +\infty$. Hence there exists $y_0 \geq 1$ such that the previous inequality holds for all $y \geq y_0$. This concludes the proof of Lemma \ref{lemme couplage}.
\end{proof}

Applying Lemma \ref{lemme couplage} with $z:=Z_i$ and $y:=y_n$ gives 

 $$
		\underset{i:E_{i;n}=1}{\max}\left|\frac{Y_{i,n}^*}{\tilde{Y}_{i,n}}-1\right|= \mathrm{O}\left(A\left(\frac{1}{\bar{F}(y_n)}\right)\right).
 $$

Now by construction of $(X_{i,n},Y_{i,n}^*)$, we see that (\ref{X_X*_bound}) is a consequence of the following lemma.
\begin{lemma} Under conditions (\ref{second order 1}) and (\ref{second order 2}), we have 
$$
||P_{X|Y > y}-\sigma(x)\mathbf{P}_X(dx)||_{TV}=O\left(A\left(\frac{1}{\bar{F}(y)}\right)\right), \text{ as }y \rightarrow +\infty.
$$

\end{lemma}

\begin{proof}
For $B \in \mathbb{R}^d$, we have
$$
\begin{array}{ccl}
|P(X \in B | Y >y)-\int_B\sigma(x)\mathbf{P}_X(dx)|&=&\left|\frac{\int_B\bar{F}_x(y)P_X(dx)}{\bar{F}(y)}-\int_B\sigma(x)\mathbf{P}_X(dx)\right| \\
                                          &\leq& \int_B\left|\frac{\bar{F}_x(y)}{\bar{F}(y)}-\sigma(x)\right|\mathbf{P}_X(dx) \\
                                          &\leq& \int_{\mathbb{R}^d}\left|\frac{\bar{F}_x(y)}{\bar{F}(y)}-\sigma(x)\right|\mathbf{P}_X(dx) \\
                                          &\leq& \int_{\mathbb{R}^d}MA\left(\frac{1}{\bar{F}(y)}\right)\mathbf{P}_X(dx) \\
                                          &=&MA\left(\frac{1}{\bar{F}(y)}\right).
\end{array}
$$
\end{proof}
This concludes the proof of Proposition \ref{coupling propertie}.

\end{proof}
Since, for each $n$ the law of $(\tilde{X}_{i,n},\tilde{Y}_{_,n})_{i=1,...,n}$ is $\mathbf{P}_{X,Y}^{\otimes n}$, we will write them, simply $(X_i,Y_i)_{i=1,...,n}$ to unburden notations.

\section{Proofs}\label{sec:proofs}
\subsection{Proof of Theorem \ref{EdHZ}}\label{Preuve du theoreme}

\subsubsection{Proof when $\mathbf{y}_n=y_n $ is deterministic}\label{prof for y_n det}
Fix $0< \varepsilon < \frac{1}{2}$ and $0 < \beta < \alpha \varepsilon/2$. The proof is divided into two steps. We first prove the result for the sample $(X_{i,n}^*,Y_{i,n}^*)_{1 \leq i \leq n}$. Then, the coupling properties of Proposition (\ref{coupling propertie}) will allow us to deduce the theorem for $(X_i,Y_i)_{1 \leq i \leq n}$.
\newline
We consider the empirical process defined for every $x\in \mathbb{R}^d$ and $y > 1$ as 
\[
\mathbb{G}_n(x,y):=\sqrt{np_n}(\mathbb{F}_n(x,y)-\mathbb{F}(x,y)), \ \text{with}
\]
$$
\mathbb{F}_n(x,y):=\frac{1}{N_n}\underset{i=1}{\overset{n}{\sum}}\mathds{1}_{\{X_i\leq x\}}\mathds{1}_{\{Y_i/y_n >y\}}E_{i,n}, \text{ and}
$$

$$
\mathbb{F}(x,y):=C(x)y^{-\alpha}.
$$
So that $\mathbb{G}_n$ almost surely satisfies $\|\mathbb{G}_n\|_{\infty,\beta}< \infty$, with 
\[
\|f\|_{\infty,\beta}:= \underset{x \in \mathbb{R}^d , y \geq 1}{\sup}|y^\beta f(x,y)|.
\]
Denote by $L^{\infty,\beta}(\mathbb{R}^d\times [1,\infty))$ the (closed) subspace of $L^{\infty}(\mathbb{R}^d\times[1,\infty))$ of all $f$ satisfying $\|f\|_{\infty,\beta} < +\infty$, such that $f(\infty,y):=\underset{\|x\|\rightarrow \infty}{\lim}f(x,y)$ exists for each $y\geq 1$, and such that $\{ y \mapsto f(\infty,y) \}$ is Càdlàg (see e.g. \cite{billingsley2} p.121). In the sequel we shall also set $\mathbb{G}_n(x,y):=0$ when $y \leq 1$.\\
First note that $\widehat{C}_n-C$ and $\hat{\gamma}_n-\gamma$ are obtained from $\mathbb{G}_n$ as images of $\mathbb{G}_n$ by the following map $\varphi$. 
\[
\begin{array}{cccl}
\varphi:& L^{\infty,\beta}(\mathbb{R}^d\times [1,\infty))&\rightarrow&L^\infty(\mathbb{R}^d)\times \mathbb{R}\\
&f&\mapsto&\left(\{x \mapsto f(x,1)\},\int_1^\infty y^{-1}f(\infty,y)\mathrm{d}y\right).
\end{array}
\]
We can remark that $\varphi$ is continuous since $\beta >0$.
\newline
By the continuous mapping theorem, we hence see that Theorem \ref{EdHZ} will be a consequence of $\mathbb{G}_n\overset{\mathscr{L}}{\rightarrow} \mathbb{W}$ in the space $L^{\infty,\beta}(\mathbb{R}^d\times [1,\infty))$ , where $\mathbb{W}$ is the centered Gaussian process with covariance
\[
\mathrm{cov}\big(\mathbb{W}(x_1,y_1),\mathbb{W}(x_2,y_2)\big)=C(x_1\wedge x_2)(y_1\vee y_2)^{-\alpha}-C(x_1)C(x_2)(y_1 y_2)^{-\alpha}.
\]
For technical reasons, we shall consider the natural extension $\mathbb{W}(x,y)\equiv 0$ for $y \in [1/2,1]$.
 
\underline{Step 1:} Define
$$
\mathbb{F}_n^*(x,y):=\frac{1}{N_n}\underset{i=1}{\overset{n}{\sum}}\mathds{1}_{\{X_i^* \leq x\}}\mathds{1}_{\{Y_{i,n}^*/y_n >y\}}E_{i,n} \  x \in \mathbb{R}^d, \ y\geq 1.
$$
The following proposition states the asymptotic behavior of the empirical cdf based on a sample of the limit model $Q$.
\begin{proposition}\label{proposition step 1}
If (\ref{second order 1}) and (\ref{second order 2}) hold, then  
$$
\mathbb{G}_n^*:=\sqrt{np_n}(\mathbb{F}_n^*-\mathbb{F})\overset{\mathscr{L}}{\rightarrow} \mathbb{W}, \text{ in } L^{\infty,\beta}(\mathbb{R}^d\times [1/2,\infty)).
$$

\end{proposition}

\begin{proof} 
Since both $\mathbb{G}_n^*$ and $\mathcal{W}$ are a.s identically zero on $[1/2,1]$, it is sufficient to prove the weak convergence in $L^{\infty,\beta}(\mathbb{R}^d\times [1,\infty))$. Since $(X_{i,n}^*,Y_{i,n}^*)_{1 \leq i \leq n}$ is independent of $(E_{i,n})_{1 \leq i \leq n}$, Lemma \ref{dony_einmahl} entails the following equality in laws
\[
\sum_{i=1}^n\delta_{X_{i,n}^*,\frac{Y_{i,n}^*}{y_n}}E_{i,n}\overset{\mathscr{L}}{=}\sum_{i=1}^{\nu(n)}\delta_{X^*_{i,n},\frac{Y^*_{i,n}}{y_n}},
\]
where $\nu(n) \sim \mathcal{B}(n,p_n)$ is independent of $(X_{i,n}^*,Y_{i,n}^*)_{1 \leq i \leq n}$. \newline
Since $\mathscr{L}(X_{i,n}^*,Y_{i,n}^*/y_n)=\sigma(x)\mathbf{P}_{X}(dx)\otimes Pareto(\alpha)=Q$, and since $\nu(n) \overset{\mathbb{P}}{\rightarrow} \infty$, $\nu(n)/np_n \overset{\mathbb{P}}{\rightarrow} 1$ and $\nu(n)$ independent of $(X_{i,n}^*,Y_{i,n}^*)_{1\leq i \leq n}$, we see that $\mathbb{G}_n\overset{\mathscr{L}}{\rightarrow}\mathbb{W}$ will be a consequence of
\[
\sqrt{k}\left(\frac{1}{k}\sum_{i=1}^k\mathds{1}_{\{U_i\leq ., V_i >.\}}-\mathbb{F}(.,.)\right)\overset{\mathscr{L}}{\underset{k \rightarrow +\infty}{\longrightarrow}} \mathbb{W} \text{ in }\mathcal{L}^{\infty,\beta}\big(\mathbb{R}^d\times [1,\infty)\big).
\]
Consider the class of functions, 
\[
\mathcal{F}_{\beta}=\big\{ f_{x,y}:(u,v)\mapsto y^{\beta}\mathds{1}_{(-\infty,x]}(u)\mathds{1}_{[y,+\infty)}(v), x \in \mathbb{R}^d, y \geq 1 \big\}.
\]
Using the isometry:
\[
\begin{array}{rcl}
L^{\infty,\beta}(\mathbb{R}^d\times[1,\infty)) &\rightarrow&L^\infty(\mathcal{F}_{\beta})\\ 
g &\mapsto &\lbrace \Psi: f_{x,y} \mapsto g(x,y) \rbrace,
\end{array}
\]
it is enough to prove that the abstract empirical process indexed by $\mathcal{F}_{\beta}$ converges weakly to the $Q$- Brownian bridge indexed by $\mathcal{F}_{\beta}$. In other words, we need to verify that $\mathcal{F}_{\beta}$ is $Q$-Donsker.
This property can be deduced from two remarks: 

\begin{enumerate}
\item $\mathcal{F}_\beta$ is a VC-subgraph class of function (see, e.g, Van der Vaart and Wellner \cite{vdvjaune}). To see this, note that 
$$\mathcal{F}_\beta \subset \big\{ f_{x,s,z}:(u,v)\mapsto z \mathds{1}_{(-\infty,x]}(u)\mathds{1}_{[y,\infty)}(v), x \in \mathbb{R}^d, s \in [1;\infty), z \in\mathbb{R} \big\}$$
which is a VC-subgraph class: the subgraph of each of its members is an hypercube of $\mathbb{R}^{d+2}$.
\item $\mathcal{F}_\beta$ has a square integrable envelope $F$. This is proved by noting that for fixed $(u,v)\in \mathbb{R}^d\times[1,+\infty)$.
\[
F^2(u,v)=\underset{x \in \mathbb{R}^d ,\  y \geq 1}{\sup}\ y^{2\beta}\mathds{1}_{[0,x]}(u)\mathds{1}_{[y,\infty)}(v)=v^{2\beta}
\]
as a consequence $F^2$ is $Q$-integrable since $\beta <\alpha/2$. 
\end{enumerate}
This concludes the proof of Proposition \ref{proposition step 1}.
\end{proof}

\underline{Step 2:} We show here that the two empirical processes $\mathbb{G}_n$ and $\mathbb{G}^*_n$ must have the same weak limit, by proving the next proposition.

\begin{proposition}\label{prop:setp 2}
Under Assumptions (\ref{second order 1}) and (\ref{second order 2}), we have 

$$
\underset{x \in \mathbb{R}^d, \ y \geq 1 }{\sup}y^{\beta}\sqrt{np_n}|\mathbb{F}_n^*(x,y)-\mathbb{F}_n(x,y)|=o_\mathbb{P}(1).
$$

\end{proposition}

\begin{proof}
Adding and subtracting
$$
\mathbb{F}_n^\sharp(x,y):=\frac{1}{N_n}\sum_{i=1}^n \mathds{1}_{\{X_i \leq x\}}\mathds{1}_{\{Y^*_{i,n}/y_n >y\}}E_{i,n}
$$
in $|\mathbb{F}_n(x,y)-\mathbb{F}_n^*(x,y) |$, the triangle inequality entails, almost surely 
\begin{align*}
&|\mathbb{F}_n(x,y)-\mathbb{F}_n^*(x,y) | \\
=& |\mathbb{F}_n(x,y)-\mathbb{F}_n^\sharp(x,y)+\mathbb{F}_n^\sharp(x,y)-\mathbb{F}_n^*(x,y) | \\
                                    \leq& \frac{1}{N_n}\underset{i=1}{\overset{n}{\sum}}|\mathds{1}_{\{X_i \leq x\}}-\mathds{1}_{\{X^*_{i,n} \leq x\}}|\mathds{1}_{\left\lbrace\frac{Y^*_{i,n}}{y_n} >y\right\rbrace }E_{i,n} \\
                                    & +\frac{1}{N_n}\underset{i=1}{\overset{n}{\sum}}|\mathds{1}_{\left\lbrace\frac{Y_i}{y_n} >y \right\rbrace}-\mathds{1}_{\left\lbrace\frac{Y^*_{i,n}}{y_n} >y \right\rbrace}| \mathds{1}_{\{X_i \leq x\}}E_{i,n} \\           
                                  \leq& \frac{1}{N_n}\underset{i=1}{\overset{n}{\sum}}\mathds{1}_{\{X_i \neq X_{i,n}^*\}}\mathds{1}_{\left\lbrace\frac{Y^*_{i,n}}{y_n} >y \right\rbrace}E_{i,n}+\frac{1}{N_n}\underset{i=1}{\overset{n}{\sum}}|\mathds{1}_{\left\lbrace\frac{Y_i}{y_n} >y \right\rbrace}-\mathds{1}_{\left\lbrace\frac{Y^*_{i,n}}{y_n} >y\right\rbrace }|E_{i,n}. \\
\end{align*}
Let us first focus on the first term. 
First, notice that
$$
\begin{array}{rl}
\underset{x \in \mathbb{R}^d,\ y \geq 1}{\sup}&\frac{y^\beta\sqrt{np_n}}{N_n}\underset{i=1}{\overset{n}{\sum}}\mathds{1}_{\{X_i \neq X_{i,n}^*\}}\mathds{1}_{\left\lbrace\frac{Y^*_{i,n}}{y_n} >y \right\rbrace}E_{i,n} \\
				=&\underset{y \geq 1}{\sup}\frac{y^\beta\sqrt{np_n}}{N_n}\underset{i=1}{\overset{n}{\sum}}\mathds{1}_{\{X_i \neq X_{i,n}^*\}}\mathds{1}_{\left\lbrace\frac{Y^*_{i,n}}{y_n} >y \right\rbrace}E_{i,n} \\
				\leq &\underset{y \geq 1}{\sup}\frac{y^\beta\sqrt{np_n}}{N_n}\left(\underset{i=1,...,n}{\max}\mathds{1}_{\left\lbrace\frac{Y^*_{i,n}}{y_n} >y \right\rbrace}E_{i,n}\right) \underset{i=1}{\overset{n}{\sum}}\mathds{1}_{\{X_i \neq X_{i,n}^*\}}E_{i,n}.\\
\end{array}
$$
Now notice that
$$
\begin{array}{ccl}
\underset{y \geq 1}{\sup} \ \underset{i=1,...,n}{\max} \ y^\beta\mathds{1}_{\left\lbrace\frac{Y^*_{i,n}}{y_n} >y\right\rbrace }E_{i,n}&=&\underset{i=1,...,n}{\max} \ \underset{y \geq 1}{\sup} \ y^\beta\mathds{1}_{[1;Y_{i,n}^*/y_n]}(y)E_{i,n} \\
                        &=& \underset{i=1,...,n}{\max}\left(\frac{Y_{i,n}^*}{y_n}\right)^\beta E_{i,n}.
\end{array}
$$
By independence between $E_{i,n}$ and $Y_{i,n}^*/y_n$, Lemma \ref{dony_einmahl} gives $$ \underset{i=1,...,n}{\max}\left(\frac{Y_{i,n}^*}{y_n}\right)^\beta E_{i,n}\overset{\mathscr{L}}{=}\underset{i=1,...,\nu(n)}{\max}\left(\frac{Y_{i,n}^*}{y_n}\right)^\beta$$ 
where $Y_{i,n}^*/y_n$ in the right hand side have a Pareto($\alpha$) distribution, whence
\begin{equation}\label{truc intermediaire a}
\underset{i=1,...,n}{\max}\left(\frac{Y_{i,n}^*}{y_n}\right)^\beta E_{i,n}=O_{\mathbb{P}}(\nu(n)^{\beta/\alpha})= O_{\mathbb{P}}((np_n)^{\beta/\alpha}).
\end{equation}
Moreover, writing $A_n:=A(1/p_n)$ one has
$$
\mathbb{E}\left(\sum_{i=1}^n\mathds{1}_{\{X_i \neq X_{i,n}^*\}}E_{i,n}\right)= np_nA_n,
$$
which entails
\begin{equation}\label{truc intermediaire b}
\sum_{i=1}^n\mathds{1}_{\{X_i \neq X_{i,n}^*\}}E_{i,n}(np_nA_n)^{-1}=O_\mathbb{P}(1).
\end{equation}
As a consequence 
\begin{align*}
&\frac{\sqrt{np_n}}{N_n}\underset{i=1,...,n}{\max}\left(\frac{Y_{i,n}^*}{y_n}\right)^\beta E_{i,n}\left(\underset{i=1}{\overset{n}{\sum}}\mathds{1}_{\{X_i \neq X_{i,n}^*\}}E_{i,n}\right)\\
&=\frac{np_n}{N_n}\underset{i=1,...,n}{\max}\left(\frac{Y_{i,n}^*}{y_n}\right)^\beta E_{i,n}\frac{\underset{i=1}{\overset{n}{\sum}}\mathds{1}_{\{X_i \neq X_{i,n}^*\}}E_{i,n}}{np_nA_n}\sqrt{np_n}A_n\\
&=O_\mathbb{P}(1)O_\mathbb{P}((np_n)^{\beta/\alpha})O_\mathbb{P}(1)\sqrt{np_n}A_n,\text{ by } (\ref{truc intermediaire a}) \text{ and } (\ref{truc intermediaire b}) \\
&=o_{\mathbb{P}}(1), \text{ by assumption of Theorem \ref{EdHZ}},\text{ since } \frac{\beta}{\alpha} <\frac{\varepsilon}{2} .
\end{align*}
Let us now focus on the convergence
$$
\underset{x \in \mathbb{R}^d, \ y \geq 1 }{\sup}y^\beta\sqrt{np_n}\frac{1}{N_n}\underset{i=1}{\overset{n}{\sum}}\left|\mathds{1}_{\left\lbrace\frac{Y_i}{y_n} >y \right\rbrace}-\mathds{1}_{\left\lbrace\frac{Y^*_{i,n}}{y_n} >y \right\rbrace}\right|E_{i,n}\overset{\mathbb{P}}{\rightarrow}0.
$$
We deduce from Proposition \ref{coupling propertie} that, almost surely:

$$
(1-MA_n)\frac{Y_i}{y_n}E_{i,n} \leq \frac{Y_{i,n}^*}{y_n}E_{i,n} \leq (1+MA_n)\frac{Y_i}{y_n}E_{i,n}.
$$
Which entails, almost surely, for all $y \geq 1$:

$$
E_{i,n}\mathds{1}_{\left\lbrace\frac{Y_{i,n}^*}{y_n} \geq (1+MA_n)y \right\rbrace}\leq E_{i,n}\mathds{1}_{\left\lbrace\frac{Y_i}{y_n} \geq y \right\rbrace} \leq E_{i,n}\mathds{1}_{\left\lbrace\frac{Y_{i,n}^*}{y_n} \geq (1-MA_n)y \right\rbrace},
$$
implying
$$
\left|\mathds{1}_{\left\lbrace\frac{Y_i}{y_n} >y\right\rbrace}-\mathds{1}_{\left\lbrace\frac{Y_{i,n}^*}{y_n} > y\right\rbrace}\right|E_{i,n} \leq \left|\mathds{1}_{\left\lbrace\frac{Y_{i,n}^*}{y_n}>(1-MA_n)y \right\rbrace}-\mathds{1}_{\left\lbrace\frac{Y_{i,n}^*}{y_n}>(1+MA_n)y \right\rbrace}\right|E_{i,n}.
$$
This implies

$$
\begin{array}{rl}
\underset{x \in \mathbb{R}^d,\ y \geq 1 }{\sup}&y^\beta\sqrt{np_n}\frac{1}{N_n}\underset{i=1}{\overset{n}{\sum}}\left|\mathds{1}_{\left\lbrace\frac{Y_i}{y_n} >y \right\rbrace}-\mathds{1}_{\left\lbrace\frac{Y^*_{i,n}}{y_n} >y\right\rbrace }\right|E_{i,n} \\
\leq& \underset{x \in \mathbb{R}^d,\ y \geq 1 }{\sup}y^\beta\sqrt{np_n}\left|\mathbb{F}_n^*(+\infty,(1-MA_n)y)-\mathbb{F}_n^*(+\infty,(1+MA_n)y)\right|.
\end{array}
$$
Consequently we have, adding and substracting expectations:
\begin{equation}
\begin{array}{rl}
\underset{x \in \mathbb{R}^d,\ y \geq 1 }{\sup}& y^\beta\sqrt{np_n}\frac{1}{N_n}\underset{i=1}{\overset{n}{\sum}}\left|\mathds{1}_{\{\frac{Y_i}{y_n} >y\} }-\mathds{1}_{\{\frac{Y^*_{i,n}}{y_n} >y\} }\right| E_{i,n} \\
                                       \leq   &\underset{x \in \mathbb{R}^d,\ y \geq 1 }{\sup}y^\beta\left|\tilde{\mathbb{G}}_n^*((1-MA_n)y)-\tilde{\mathbb{G}}_n^*((1+MA_n)y)\right| \\
                                              &+\sqrt{np_n}\underset{\ y \geq 1 }{\sup} \ y^\beta(\bar{F}_{Pareto(\alpha)}((1-MA_n)y)-\bar{F}_{Pareto(\alpha)}((1+MA_n)y)),
\end{array}
\end{equation}
where we wrote $\tilde{\mathbb{G}}_n^*(y)=\mathbb{G}_n^*(+\infty,y)$ and where $\bar{F}_{Pareto(\alpha)}(y)=y^{-\alpha}\mathds{1}_{[1,+\infty )}(y)$.
For, $y\geq 1$ we can bound
\[
y^\beta(\bar{F}_{Pareto(\alpha)}((1-MA_n)y)-\bar{F}_{Pareto(\alpha)}((1+MA_n)y))
\]
by the quantity 
\[
y^\beta |1-((1+MA_n)y)^{-\alpha}|\mathds{1}_{\{(1-MA_n)y<1\}}+y^\beta |((1-MA_n)y)^{-\alpha}-((1+MA_n)y)^{-\alpha}|\mathds{1}_{\{(1-MA_n)y\geq 1\}}.
\]
In the first term, since $(1-MA_n)y<1$, we can write 
\begin{align*}
&y^\beta |1-((1+A_n)y)^{-\alpha}|\mathds{1}_{\{(1-MA_n)y<1\}} \\
&\leq y^{\beta-\alpha} |y^\alpha-(1+A_n)^{-\alpha}|\mathds{1}_{\{(1-MA_n)y<1\}}\\
&\leq y^{\beta-\alpha} |(1-MA_n)^{-\alpha}-(1+A_n)^{-\alpha}|\mathds{1}_{\{(1-MA_n)y<1\}}\\
&\leq 4\alpha MA_n, \ \text{since }\beta-\alpha <0.
\end{align*}
Now, we can similarly study the second term, 
\begin{align*}
&y^\beta\left|\bar{F}_{Pareto(\alpha)}((1-MA_n)y)-\bar{F}_{Pareto(\alpha)}((1+MA_n)y)\right|\mathds{1}_{\{(1-MA_n)y\geq 1\}}\\
&=y^\beta(((1-MA_n)y)^{-\alpha}-((1+MA_n)y)^{-\alpha})\mathds{1}_{\{(1-MA_n)y\geq 1\}} \\
            &=y^{\beta-\alpha}((1-MA_n)^{-\alpha}-(1+MA_n)^{-\alpha})\mathds{1}_{\{(1-MA_n)y\geq 1\}}    \\
            &\leq y^{\beta-\alpha}4\alpha M A_n \\
            &\leq 4\alpha M A_n,\text{ since } \beta-\alpha<0 .
\end{align*}

This leads to 
$$
\begin{array}{rl}
\sqrt{np_n}&\underset{x \in \mathbb{R}^d,\ y \geq 1 }{\sup}y^\beta|\bar{F}_{Pareto(\alpha)}((1-MA_n)y)-\bar{F}_{Pareto(\alpha)}((1+MA_n)y)| \\
&\leq 8\alpha M \sqrt{np_n}A_n,
\end{array}
$$
which converges in probability to $0$ by assumptions of Theorem \ref{EdHZ}. \\
Now, by Proposition \ref{proposition step 1}, the continuous mapping theorem together with the Portmanteau theorem entail :
\begin{align*}
\forall \varepsilon >0, \forall \rho >0, \ \overline{\lim}&\ \mathbb{P}\left(\underset{y \geq 1, \delta < \rho}{\sup} \ y^\beta |\tilde{\mathbb{G}}_n^*((1-\delta)y)-\tilde{\mathbb{G}}_n^*((1+\delta)y)| \geq  \varepsilon\right) \\ \leq & \mathbb{P}\left(\underset{y \geq 1, \delta < \rho}{\sup} \ y^\beta |\tilde{\mathbb{W}}((1-\delta)y)-\tilde{\mathbb{W}}((1+\delta)y)| \geq \varepsilon\right). 
\end{align*}
Where $\tilde{\mathbb{W}}(y):=\mathbb{W}(+\infty,y)$ is the centered Gaussian process with covariance function
\[
cov(\tilde{\mathbb{W}}(y_1),\tilde{\mathbb{W}}(y_2)):=(y_1\vee y_2)^{-\alpha}-(y_1y_1)^{-\alpha}, \ (y_1,y_2) \in [1,+\infty)^2.
\]
The proof of Proposition \ref{prop:setp 2} will hence be concluded if we establish the following lemma:

\begin{lemma}We have 
\[
\underset{y \geq 1, \delta<\rho}{\sup} \ y^\beta |\tilde{\mathbb{W}}((1-\delta)y)-\tilde{\mathbb{W}}((1+\delta)y)|\underset{\rho \rightarrow 0}{\overset{\mathbb{P}}{\longrightarrow}}0.
\]
\end{lemma}

\begin{proof}
Let $\mathbb{B}_0$ be the standard Brownian bridge on $[0,2]$ (with $\mathbb{B}_0$ identically zero on $[1,2]$). $\tilde{\mathbb{W}}$ has the same law as $\{y \mapsto \mathbb{B}_0 (y^{-\alpha})\}$ (see \cite{SW09}, p. 99), from where

$$
\begin{array}{rl}
\underset{ y \geq 1 , \delta<\rho}{\sup} &y^\beta|\tilde{\mathbb{W}}((1-\delta)y)-\tilde{\mathbb{W}}((1+\delta)y)|\\
\overset{\mathscr{L}}{=} & \underset{ y \geq 1, \delta<\rho }{\sup} \ y^\beta \left|\mathbb{B}_0(((1-\delta)y)^{-\alpha})-\mathbb{B}_0(((1+\delta)y)^{-\alpha})\right| \\
\leq & \underset{ 0 \leq y \leq 1 , \delta<\rho}{\sup} \ y^{-\beta/\alpha}\left|\mathbb{B}_0((1-\delta)^{-\alpha}y)-\mathbb{B}_0((1+\delta)^{-\alpha}y)\right|, \ \text{almost surely}. \\

\end{array}
$$ 
Since $\beta/\alpha < 1/2$ the process $\mathbb{B}_0$ is a.s-$\beta/\alpha$-Hölder continuous. Consequently, for an a.s finite random variable $H$ one has
$$
\begin{array}{rl}
\underset{ 0\leq y \leq 1, \delta < \rho }{\sup} &y^{-\beta/\alpha}|\mathbb{B}_0((1-\delta)^{-\alpha}y)-\mathbb{B}_0((1+\delta)^{-\alpha}y)|\\
\leq&y^{-\beta/\alpha}|(1-\rho)^{-\alpha}-(1+\rho)^{-\alpha}|^{\beta/\alpha}y^{\beta/\alpha} H \\
=& |(1-\rho)^{-\alpha}-(1+\rho)^{-\alpha}|^{\beta/\alpha} H \\
=& (4\alpha \rho)^{\beta/\alpha}H.
\end{array}
$$
\end{proof}
This lemma concludes the proof of Theorem \ref{EdHZ} when $\textbf{y}_n\equiv y_n$.
\end{proof}

\subsubsection{Proof of Theorem \ref{EdHZ} in the general case.}
We now drop the assumption $\textbf{y}_n\equiv y_n$ and we relax it to $\frac{\mathbf{y}_n}{y_n}\overset{\mathbb{P}}{\rightarrow}1$ to achieve the proof of Theorem \ref{EdHZ} in its full generality. We shall use the results of \S \ref{prof for y_n det}.

Write
$$
\overset{\vee}{\mathbb{F}}_n(x,y):=\frac{1}{N_n}\underset{i=1}{\overset{n}{\sum}}\mathds{1}_{\{X_i \leq x\}}\mathds{1}_{\{Y_i/y_n >y\}} \ \text{and}
$$

\[
\overset{\vee}{\mathbb{G}}_n(x,y):=\sqrt{np_n}\left(\overset{\vee}{\mathbb{F}}_n(x,y)-\mathbb{F}(x,y)\right).
\]
Consider the followings maps $(g_n)_{n \in \mathbb{N}}$ and $g$ from $L^{\infty,\beta}\big(\mathbb{R}^d \times [1/2,\infty)\big) \times [1,\infty)$ to $L^{\infty,\beta}(\mathbb{R}^d\times[1,\infty))$ 
\[
g_n:(\varphi, u)\mapsto \sqrt{np_n}\left(\frac{\mathbb{F}(.,u.)+\frac{1}{\sqrt{np_n}}\varphi(.,u.)}{\mathbb{F}(+\infty,u)+\frac{1}{\sqrt{np_n}}\varphi(+\infty,u)}-\mathbb{F}(.,.)\right), \ \text{and}
\]
\[
g:(\varphi, u)\mapsto u^\alpha\Big(\varphi(.,u.)-\varphi(+\infty,u)\mathbb{F}(.,.)\Big). 
\]
Now write $u_n:=\frac{\mathbf{y}_n}{y_n}$. From \S \ref{prof for y_n det}, we know that $\left(\overset{\vee}{\mathbb{G}}_n, u_n\right) \overset{\mathscr{L}}{\rightarrow}\left(\mathbb{W},1\right)$ in $L^{\infty,\beta}(\mathbb{R}^d\times[1/2,\infty))$. Also note that $\mathbb{W}$ almost surely belongs to
$$\mathbb{D}_0=\left\lbrace\varphi \in L^{\infty,\beta}(\mathbb{R}^d\times[1,+\infty)), \underset{x \in \mathbb{R}^d, y,y'>1/2}{\sup}\frac{|\varphi(x,y)-\varphi(x,y')|}{|y-y'|^{\beta/\alpha}}<+\infty\right\rbrace .$$
Notice that $\mathbb{G}_n=g_n(\overset{\vee}{\mathbb{G}}_n,u_n)$ and $g(\mathbb{W},1)=\mathbb{W}$.

Using the extended continuous mapping theorem, (see, e.g Theorem 1.11.1 p. 67 in \cite{vdvjaune}) it is sufficient to check that, for any sequence $\varphi_n$ of elements of $L^{\infty,\beta}(\mathbb{R}^d\times[1/2,+\infty))$ that converges to some $\varphi \in \mathbb{D}_0$, and for any sequence $u_n\rightarrow 1$ one has $g_n(\varphi_n,u_n)\rightarrow g(\varphi,1)$ in $L^{\infty,\beta}(\mathbb{R}^d\times[1,+\infty))$. Here, convergence in $L^{\infty,\beta}(\mathbb{R}^d\times[1/2,+\infty))$ is understood as with respect to the natural extension of $\|.\|_{\infty, \beta}$ on $\mathbb{R}^d\times[1/2,+\infty)$.\\ 
For fixed $(x,y) \in \mathbb{R}^d\times [1/2,+\infty)$ and $n \geq 1$, we have
\begin{align*}
&|g_n(\varphi_n,u_n)(x,y)-g(\varphi,1)(x,y)|\\
&=\left| \sqrt{np_n}\left(\frac{\mathbb{F}(x,u_ny)+\frac{1}{\sqrt{np_n}}\varphi_n(x,u_ny)}{\mathbb{F}(+\infty,u_n)+\frac{1}{\sqrt{np_n}}\varphi_n(+\infty,u_n)}-\mathbb{F}(x,y)\right)-\Big(\varphi(x,y)-\varphi(+\infty,1)\mathbb{F}(x,y)\Big)\right|.\\
\end{align*} 
Seeing that $\mathbb{F}(x,y)/\mathbb{F}(+\infty,1)=\mathbb{F}(x,y)$ we deduce that 
\begin{align*}
&\frac{\mathbb{F}(x,u_ny)+\frac{1}{\sqrt{np_n}}\varphi_n(x,u_ny)}{\mathbb{F}(+\infty,u_n)+\frac{1}{\sqrt{np_n}}\varphi_n(+\infty,u_n)}-\mathbb{F}(x,y) \\
&=\mathbb{F}(x,y)\left(\frac{1+\frac{1}{\sqrt{np_n}}\frac{\varphi_n(x,u_ny)}{\mathbb{F}(x,u_ny)}}{1+\frac{1}{\sqrt{np_n}}\frac{\varphi_n(+\infty,u_n)}{\mathbb{F}(+\infty,u_n)}}-1\right) \\
&=\mathbb{F}(x,y)\left(\left(1+\frac{1}{\sqrt{np_n}}\frac{\varphi_n(x,u_ny)}{\mathbb{F}(x,u_ny)}\right)\left(1-\frac{u_n^\alpha}{\sqrt{np_n}}\varphi_n(+\infty,u_n)+\frac{u_n^\alpha}{\sqrt{np_n}}\varphi_n(+\infty,u_n)\epsilon_n\right)-1\right) \\
&=\mathbb{F}(x,y)\left(\frac{1}{\sqrt{np_n}}\left(\frac{\varphi_n(x,u_ny)}{\mathbb{F}(x,u_ny)}-u_n^\alpha\varphi_n(+\infty,u_n)\right)+R_n(x,y)\right), \\
\end{align*}
with $\epsilon_n\rightarrow 0$ a sequence of real numbers, not depending upon $x$ and $y$, and with
\[
R_n(x,y):=\frac{u_n^\alpha}{\sqrt{np_n}}\varphi_n(+\infty,u_n)\epsilon_n+\frac{u_n^{2\alpha}}{np_n}\varphi_n(+\infty,u_n)\frac{\varphi_n(x,u_ny)}{\mathbb{F}(x,y)}+\frac{u_n^{2\alpha}}{np_n}\varphi_n(+\infty,u_n)\epsilon_n\frac{\varphi_n(x,u_ny)}{\mathbb{F}(x,y)}.
\]
This implies that 
\[
\| g_n(\varphi_n,u_n)-g(\varphi,1)\|_{\infty,\beta} \leq B_{1,n}+B_{2,n}+B_{3,n}+B_{4,n},
\]
where the bounds, $B_{1,n},...,B_{4,n}$ are detailed below and will be proved to converge to zero as $n\rightarrow \infty$. 
\begin{itemize}
	\item $B_{1,n}:=\|u_n^\alpha\varphi_n(.,u_n.)-\varphi(.,.) \|_{\infty,\beta}$ is controlled as follows
	\begin{align*}
	\|& u_n^\alpha\varphi_n(.,u_n.)-\varphi(.,.) \|_{\infty,\beta} \\
	 \leq& \|u_n^\alpha\varphi_n(.,u_n.)-\varphi_n(.,u_n.)\|_{\infty,\beta} + \|\varphi_n(.,u_n.)-\varphi(.,.) \|_{\infty,\beta} \\
	=&|u_n^\alpha-1|\|\varphi_n(.,u_n.)\|_{\infty,\beta}+\|\varphi_n(.,u_n.)-\varphi(.,.) \|_{\infty,\beta} \\
	 \leq & |u_n^\alpha-1|\|\varphi_n(.,u_n.)\|_{\infty,\beta}+\|\varphi_n(.,u_n.)-\varphi(.,u_n.) \|_{\infty,\beta}\\
	&+\|\varphi(.,u_n.)-\varphi(.,.) \|_{\infty,\beta} \\
	 \leq & |u_n^\alpha-1|\|\varphi_n(.,u_n.)\|_{\infty,\beta}+u_n^{-\beta}\underset{x \in \mathbb{R}^d, \ y \geq 1/2}{\sup} \ y^\beta|\varphi_n(x,y)-\varphi(x,y)|\\
	&+H_\varphi |u_n-u|^{\beta/\alpha},
	\end{align*}
	where $H_\varphi:=\sup\{|y-y'|^{-\beta/\alpha}|\varphi(x,y)-\varphi(x,y')|,x \in \mathbb{R}^d, y,y'\geq 1/2 \}$ is finite since $\varphi \in \mathbb{D}_0$. \\
	The first term converges to 0 since $u_n \rightarrow 1$ and $\varphi_n$ belongs $L^{\infty,\beta}(\mathbb{R}^d\times[1/2,+\infty))$. So does the second because $\varphi_n \rightarrow\varphi$, in $L^{\infty,\beta}(\mathbb{R}^d\times[1/2,+\infty))$. And so does the third because $H_\varphi $ is finite. 
	\item $B_{2,n}:=\| (u_n^\alpha\varphi_n(+\infty,u_n)-\varphi(+\infty,1))\mathbb{F}\|_{\infty,\beta}$. Remark that, 
	\begin{align*}
	&\| (u_n^\alpha\varphi_n(+\infty,u_n)-\varphi(+\infty,1))\mathbb{F}\|_{\infty,\beta} \\
	&\leq \big(|u_n^\alpha\varphi_n(+\infty,u_n)-\varphi_n(+\infty,u_n)|+|\varphi_n(+\infty,u_n)-\varphi(+\infty,1)|\big)\|\mathbb{F}\|_{\infty,\beta}.
	\end{align*}
	Since $\beta < \alpha \varepsilon <\alpha/2$, $\|\mathbb{F}\|_{\infty,\beta}$ is finite, from where $B_{2,n}\rightarrow 0$.
	\item $B_{3,n}:=\|u_n^\alpha\varphi_n(+\infty,u_n)\epsilon_n\mathbb{F} \|_{\infty,\beta} \leq |u_n^\alpha\varphi_n(+\infty,u_n)| \times |\epsilon_n| \times \|\mathbb{F} \|_{\infty,\beta}$. Since $\|\mathbb{F}\|_{\infty,\beta}$ is finite, since $|u_n^\alpha\varphi_n(+\infty,u_n)|$ is a converging sequence, and since $|\epsilon_n|\rightarrow 0$, we conclude that $B_{3,n} \rightarrow 0$.  
	\item $B_{4,n}:=\big(1+|\epsilon_n|\big)\Big\|\frac{u_n^{2\alpha}}{\sqrt{np_n}}\varphi_n(+\infty,u_n)\varphi_n(.,u_n.) \Big\|_{\infty,\beta}\leq \big(1+|\epsilon_n|\big)\Big|\frac{u_n^{2\alpha}}{\sqrt{np_n}}\varphi_n(+\infty,u_n)\Big| \times \|\varphi_n(.,u_n.) \|_{\infty,\beta}$. Since $\varphi_n \in L^{\infty,\beta}(\mathbb{R}^d\times[1/2,+\infty))$, the same arguments as for $B_{3,n}$ entail the convergence to zero of $B_{4,n}$.
\end{itemize}

\subsection{Proof of theorem \ref{thm:normality-skedasis}}

Let $x$ be a fixed point in $\mathbb{R}^d$. We keep it fixed in all this section.
To prove the asymptotic normality of $\hat{\sigma}_n(x)$ we first establish the asymptotic normality of the numerator and the denominator separately. Note that we don't need to study their joint asymptotic normality, because only the numerator will rule the asymptotic normality of $\hat{\sigma}_n(x)$, as its rate of convergence is the slowest.

\begin{proposition}
Assume that $(p_n)_{n \geq 1}$ and $(h_n)_{n \geq 1}$ both converge to 0 and satisfy $np_nh_n^d \rightarrow 0$. We have 

\begin{equation}\label{normalite_numerateur}
\frac{1}{\sqrt{np_nh_n^d}}\sum_{i=1}^n\frac{\mathds{1}_{\{|X_i-x|\leq h_n,Y_i >\mathbf{y}_n\}}-\mathbb{P}(|X_i-x|\leq h_n,Y_i >\mathbf{y}_n)}{\sqrt{\sigma(x)f(x)}}\overset{\mathscr{L}}{\rightarrow} \mathcal{N}(0,1),
\end{equation}
\begin{equation}\label{normalite_denominateur}
\frac{1}{\sqrt{nh_n^d}}\sum_{i=1}^n\frac{\mathds{1}_{\{|X_i-x|\leq h_n\}}-\mathbb{P}(|X_i-x|\leq h_n)}{\sqrt{f(x)}}\overset{\mathscr{L}}{\rightarrow} \mathcal{N}(0,1)
\end{equation}
and
\begin{equation}\label{normalite_proportion}
\frac{1}{\sqrt{np_n}}\sum_{i=1}^n(\mathds{1}_{\{Y_i > \mathbf{y}_n\}}-p_n) \overset{\mathscr{L}}{\rightarrow} \mathcal{N}(0,1).
\end{equation}
\end{proposition}

\begin{proof}
Note that (\ref{normalite_proportion}) is the central limit theorem for binomial$(n,p_n)$ sequences with $p_n\rightarrow 0$ and $np_n\rightarrow \infty$.
 The proof of the two other convergences relies on Lindeberg-Levy Theorem (see e.g \cite{billingsley1} Theorem 27.2 p. 359). First, denote 
\[
Z_{i,n}:=\frac{\mathds{1}_{\{|X_i-x|\leq h_n,Y_i >\textbf{y}_n\}}-\mathbb{P}(|X_i-x|\leq h_n,Y_i >\textbf{y}_n)}{\sqrt{np_nh_n^d}\sqrt{\sigma(x)f(x)}}
\]
and remark that $\mathbb{E}\left(Z_{i,n}\right)=0.$
Moreover we can write, 
\[
\begin{array}{ccl}
\mathbb{E}\left(\mathds{1}_{\{|X_i-x|\leq h_n,Y_i >\textbf{y}_n\}}\right)&=&\int_{B(x,h)}\mathbb{P}(Y_i> \mathbf{y}_n|X_i=z)P_X(dz) \\
                      &\approx&\int_{B(x,h)}\sigma(z)p_nP_X(dz) \quad (a)\\
                      &\approx&\sigma(x)f(x)p_nh_n^d, \quad (b)
\end{array}
\]

where $(a)$ is a consequence of the uniformity in assumption \eqref{second order 1}, while equivalence $(b)$ holds by our assumptions upon the regularity of both $f$ and $\sigma$ in Theorem \ref{thm:normality-skedasis}. We deduce that $\sup\{|n\text{Var}\left(Z_{i,n}\right)-1|,\; i=1,...,n\}\rightarrow 0.$
Note that the result is proved, if for all $\varepsilon >0$, we have 
\[
\sum_{i=1}^n\int_{\{Z_{i,n}>\varepsilon\}}Z_{i,n}^2 P_X(dx)\rightarrow 0.
\]
This convergence holds since the set $\{Z_{i,n}>\varepsilon\}$ can be written  
\[
\{|\mathds{1}_{\{|X_i-x|\leq h_n,Y_i >\textbf{y}_n\}}-\mathbb{P}(|X_i-x|\leq h_n,Y_i >\textbf{y}_n)| \geq\varepsilon \sqrt{\sigma(x)f(x)}\sqrt{np_nh_n^d}\}
\]
which is empty when $n$ is large enough, since $np_nh_n^d\rightarrow +\infty$. We can now apply the Lindeberg-Levy Theorem to show \eqref{normalite_numerateur}. The convergences \eqref{normalite_denominateur} is proved in a similar manner way.
\end{proof}
Now, writing 
\[
\hat\sigma_n(x)=\frac{n}{\sum_{i=1}^n\mathds{1}_{\{Y_i> \mathbf{y}_n\}}}\times\frac{\sum_{i=1}^n\mathds{1}_{\{|x-X_i|< h_n\}}\mathds{1}_{\{Y_i> \mathbf{y}_n\}}}{\sum_{i=1}^n\mathds{1}_{\{|x-X_i|< h_n\}}},
\]
the following of the proof is to combine the three asymptotic normalities. We have 
\[
\hat{\sigma}_n(x)=\frac{1}{1+\frac{1}{\sqrt{np_n}}\underset{i=1}{\overset{n}{\sum}}Z_{i,n}^\sharp}\times\frac{\frac{\mathbb{P}(|X-x|\leq h_n,Y >\textbf{y}_n)}{p_nh_n^d}+\sqrt{\frac{f(x)\sigma(x)}{np_nh_n^d}}\underset{i=1}{\overset{n}{\sum}} Z_{i,n}}{\frac{\mathbb{P}(|X-x|\leq h_n)}{h_n^d}+\sqrt{\frac{f(x)}{nh_n^d}}\underset{i=1}{\overset{n}{\sum}} \tilde{Z}_{i,n}}
\] 
where 
\[
\tilde{Z}_{i,n}:=\frac{\mathds{1}_{\{|X_i-x|\leq h_n\}}-\mathbb{P}(|X_i-x|\leq h_n)}{\sqrt{f(x)}\sqrt{nh_n^d}}
\]
and
\[
Z_{i,n}^\sharp:=\frac{\mathds{1}_{\{Y_i > \textbf{y}_n\}}-p_n}{\sqrt{np_n}}.
\]
Now write 
\[
\sigma_{h_n}(x):=\frac{\mathbb{P}(|X-x|\leq h_n,Y >\textbf{y}_n)}{p_nh_n^df(x)}.
\]
Since $f$ is continuous and bounder away from zero on a neighbourhood of $x$ we have
\[
\hat{\sigma}_n(x)=\frac{1}{1+\frac{1}{\sqrt{np_n}}\underset{i=1}{\overset{n}{\sum}}Z_{i,n}^\sharp}\frac{\sigma_{h_n}(x)f(x)(1+\varepsilon_{n,1})+\sqrt{\frac{f(x)\sigma(x)}{np_nh_n^d}}\underset{i=1}{\overset{n}{\sum}} Z_{i,n}}{f(x)(1+\varepsilon_{n,2})+\sqrt{\frac{f(x)}{nh_n^d}}\underset{i=1}{\overset{n}{\sum}} \tilde{Z}_{i,n}},
\]
with |$\varepsilon_{n,1}|\vee|\varepsilon_{n,2}|\rightarrow 0$. Now, factorizing by $f(x)$, the Taylor expansion of the denominator gives 
\begin{align*}
\hat{\sigma}_n(x)=\frac{1}{1+\frac{1}{\sqrt{np_n}}\underset{i=1}{\overset{n}{\sum}}Z_{i,n}^\sharp}\left(\sigma_{h_n}(x)+\sqrt{\frac{\sigma(x)}{np_nh_n^df(x)}}\underset{i=1}{\overset{n}{\sum}} Z_{i,n}\right)\\ \times \left(1-\sqrt{\frac{1}{nh_n^df(x)}}\underset{i=1}{\overset{n}{\sum}} \tilde{Z}_{i,n}+o_{\mathbb{P}}\left(\sqrt{\frac{1}{nh_n^df(x)}}\right)\right).
\end{align*} 

Since $f(x)$ bounded away from zero on a neighborhood of $x$, and remarking that $(nh_n^d)^{-1}=o\Big((np_nh_n^d)^{-1}\Big)$, we have, by \eqref{normalite_numerateur} and \eqref{normalite_denominateur}, 
\[
\hat{\sigma}_n(x)=\frac{1}{1+\frac{1}{\sqrt{np_n}}\underset{i=1}{\overset{n}{\sum}}Z_{i,n}^\sharp}\left(\sigma_{h_n}(x)+\sqrt{\frac{\sigma(x)}{np_nh_n^df(x)}}\underset{i=1}{\overset{n}{\sum}} Z_{i,n}+o_{\mathbb{P}}\left(\frac{1}{\sqrt{np_nh_n^d}}\right)\right).
\] 
Moreover, with one more Taylor expansion of the denominator, we have, by \eqref{normalite_proportion},
\[
\hat{\sigma}_n(x)=\sigma_{h_n}(x)+\sqrt{\frac{\sigma(x)}{np_nh_n^df(x)}}\underset{i=1}{\overset{n}{\sum}} Z_{i,n}+o_{\mathbb{P}}\left(\frac{1}{\sqrt{np_nh_n^d}}\right),
\] 
which entails 
\[
\sqrt{np_nh_n^d}(\hat{\sigma}_n(x)-\sigma_{h_n}(x))=\sqrt{\frac{\sigma(x)}{f(x)}}\underset{i=1}{\overset{n}{\sum}} Z_{i,n}+o_{\mathbb{P}}(1).
\]
The asymptotic normality of $\underset{i=1}{\overset{n}{\sum}} Z_{i,n}$ gives 
\[
\sqrt{np_nh_n^d}(\hat{\sigma}_n(x)-\sigma_{h_n}(x))\overset{\mathscr{L}}{\rightarrow} \mathcal{N}\left(0,\frac{\sigma(x)}{f(x)}\right).
\]
The proof is achieved by noticing that assumption \eqref{second order 1} entails
\begin{align*}
\sqrt{np_nh_n^d}|\sigma_{h_n}(x)-\sigma(x)|&=\sqrt{np_nh_n^d}\left|\frac{\mathbb{P}(|X-x|\leq h_n,Y > \textbf{y}_n)}{f(x)h_n^d\mathbb{P}(Y>\textbf{y}_n)}-\sigma(x)\right| \\
&=\sqrt{np_nh_n^d}\left|\frac{\mathbb{P}(Y > \textbf{y}_n|X\in B(x,h_n))}{\mathbb{P}(Y>\textbf{y}_n)}-\sigma(x)\right| \\
&=O\left(\sqrt{np_nh_n^d}A\left(\frac{1}{p_n}\right)\right)\rightarrow 0.
\end{align*}

\subsection{Proof of theorem \ref{thm:quantile-normality}}
For sake of clarity, we first express conditions \eqref{second order 1} and \eqref{second order 2} in terms of the tail quantile function $U$: we have, uniformly in $x$,
\[
\left|\frac{U_x(1/\alpha_n)}{U(\sigma(x)/\alpha_n)}-1\right|=O(A_n) \text{ and }\left|\frac{U(1/\alpha_n)}{xU(x^{-\alpha}/\alpha_n)}-1\right|=O(A_n).
\]

Start the proof by splitting the quantity of interest into four parts,  
\begin{align*}
\log\left(\frac{\hat{q}(\alpha_n|x)}{q(\alpha_n|x)}\right)=&\log\left(\frac{\textbf{y}_n}{q(\alpha_n|x)}\left(\frac{\hat{p}_n\hat{\sigma}_n(x)}{\alpha_n}\right)^{\hat{\gamma}_n}\right) \\
=&\log\left(\frac{\textbf{y}_n}{q(\alpha_n|x)}\left(\frac{p_n\hat{\sigma}_n(x)}{\alpha_n}\right)^{\hat{\gamma}_n}\left(\frac{\hat{p}_n}{p_n}\right)^{\hat{\gamma}_n}\right)\\
=&\log\left(\frac{\textbf{y}_n}{q(\alpha_n|x)}\right)+\hat{\gamma}_n\log\left(\frac{p_n}{\alpha_n}\right)+\hat{\gamma}_n\log(\hat{\sigma}_n(x))+\hat{\gamma}_n\log\left(\frac{\hat{p}_n}{p_n}\right) \\
=&\log\left(\frac{\textbf{y}_n}{q(\alpha_n|x)}\left(\frac{p_n}{\alpha_n}\right)^\gamma\right)+(\hat{\gamma}_n-\gamma)\log\left(\frac{p_n}{\alpha_n}\right)\\
&+\hat{\gamma}_n\log(\hat{\sigma}_n(x))+\hat{\gamma}_n\log\left(\frac{\hat{p}_n}{p_n}\right). 
\end{align*}
Moreover we can see that 
\begin{align*}
\log\left(\frac{\textbf{y}_n}{q(\alpha_n|x)}\left(\frac{p_n}{\alpha_n}\right)^\gamma\right)=&\log\left(\frac{U(1/p_n)}{U_x(1/\alpha_n)}\left(\frac{p_n}{\alpha_n}\right)^\gamma\right)\\
=&\log\left(\frac{U(1/p_n)}{U(1/\alpha_n)}\left(\frac{p_n}{\alpha_n}\right)^\gamma\right)+\log\left(\frac{U(1/\alpha_n)}{U_x(1/\alpha_n)}\right)
\end{align*}
Then, we write 
\[
\frac{\sqrt{np_n}}{\log(p_n/\alpha_n)}\log\left(\frac{\hat{q}(\alpha_n|x)}{q(\alpha_n|x)}\right)=Q_{1,n}+Q_{2,n}+Q_{3,n}+Q_{4,n},
\]
with 
\begin{align*}
Q_{1,n}=&\frac{\sqrt{np_n}}{\log(p_n/\alpha_n)}\log\left(\frac{U(1/p_n)}{U(1/\alpha_n)}\left(\frac{p_n}{\alpha_n}\right)^\gamma\right), \\
Q_{2,n}=&\sqrt{np_n}(\hat{\gamma}_n-\gamma),\\
Q_{3,n}=&\frac{\sqrt{np_n}}{\log(p_n/\alpha_n)} \left(\hat{\gamma}_n\log(\hat{\sigma}_n(x))+\log\left(\frac{U(1/\alpha_n)}{U_x(1/\alpha_n)}\right)\right),\\
Q_{4,n}=&\frac{\sqrt{np_n}}{\log(p_n/\alpha_n)}\hat{\gamma}_n\log\left(\frac{\hat{p}_n}{p_n}\right).
\end{align*}
First, condition \eqref{second order 2} entails
\begin{align*}
Q_{1,n}\sim &\frac{\sqrt{np_n}}{\log(p_n/\alpha_n)}\left(\frac{U(1/p_n)}{U(1/\alpha_n)}\left(\frac{p_n}{\alpha_n}\right)^\gamma-1\right)\\
\sim &\frac{\sqrt{np_n}}{\log(p_n/\alpha_n)}\left(\frac{U((\alpha_n/p_n)^{\alpha\gamma}/\alpha_n)}{U(1/\alpha_n)}\left(\frac{p_n}{\alpha_n}\right)^\gamma-1\right)\\
=&\frac{\sqrt{np_n}}{\log(p_n/\alpha_n)}O(A_n).
\end{align*}
Since $\alpha_n=o(p_n)$, we see $\log(p_n/\alpha_n)^{-1}\rightarrow 0$ together with $\sqrt{np_n}A_n\rightarrow 0$ entails that $Q_{1,n}\rightarrow 0$. \\
Second, we know by Theorem \ref{EdHZ} that $Q_{2,n}\overset{\mathscr{L}}{\rightarrow}\mathcal{N}(0,\gamma^2)$. \\
Now $Q_{3,n}$ is studied remarking that 
\[
\log\left(\frac{U(1/\alpha_n)}{U_x(1/\alpha_n)}\right)=\log\left(\frac{U(\sigma(x)/\alpha_n)}{U_x(1/\alpha_n)}\right)+\log\left(\frac{U(1/\alpha_n)}{\sigma(x)^{-\gamma}U(\sigma(x)/\alpha_n)}\right)-\gamma\log(\sigma(x)).
\]
Together with \eqref{second order 1} and \eqref{second order 2}, one has
\[
\log\left(\frac{U(1/\alpha_n)}{U_x(1/\alpha_n)}\right)=O(A_n)-\gamma\log(\sigma(x)).
\] 
Consequently, 
\[
Q_{3,n}=\frac{\sqrt{np_n}}{\log(p_n/\alpha_n)}O(A_n)+\frac{\sqrt{np_n}}{\log(p_n/\alpha_n)}\Big(\hat{\gamma}_n\log(\hat{\sigma}_n(x))-\gamma\log(\sigma(x))\Big).
\]
Hence, the asymptotic behavior of $Q_{3,n}$ is ruled by that of $\hat{\gamma}_n\log(\hat{\sigma}_n(x))-\gamma\log(\sigma(x))$, which we split into
\[
(\hat{\gamma}_n-\gamma)\log(\hat{\sigma}_n(x))+\gamma\log(\hat{\sigma}_n(x))-\gamma\log(\sigma(x)).
\]
Now, Theorem \ref{EdHZ} entails that 
\[
\frac{\log(\hat{\sigma}_n(x))}{\log(p_n/\alpha_n)}\sqrt{np_n}(\hat{\gamma}_n-\gamma)\overset{\mathbb{P}}{\rightarrow} 0.
\]
Moreover, Theorem \ref{thm:normality-skedasis} together with the delta-method show that 
\begin{align*}
\frac{\sqrt{np_n}}{\log(p_n/\alpha_n)}&(\gamma\log(\hat{\sigma}_n(x))-\gamma\log(\sigma(x)))\\=&\frac{\sqrt{np_nh_n^d}}{\sqrt{h_n^d}\log(p_n/\alpha_n)}(\gamma\log(\hat{\sigma}_n(x))-\gamma\log(\sigma(x)))\overset{\mathbb{P}}{\rightarrow} 0.
\end{align*}
Now, using the notation introduced in the proof of Theorem \ref{thm:normality-skedasis}, we have 
\begin{align*}
\frac{\sqrt{np_n}}{\log(p_n/\alpha_n)}\log\left(\frac{\hat{p}_n}{p_n}\right)=&\frac{\sqrt{np_n}}{\log(p_n/\alpha_n)}\log\left(1+\frac{1}{\sqrt{np_n}}\sum_{i=1}^nZ_{i,n}^\sharp\right)\\
\sim& \frac{1}{\log(p_n/\alpha_n)}\sum_{i=1}^nZ_{i,n}^\sharp+o_\mathbb{P}\left(\frac{1}{\log(p_n/\alpha_n)}\right) \\
&\overset{\mathbb{P}}{\rightarrow}0,
\end{align*}
which proves that $Q_{4,n}\overset{\mathbb{P}}{\rightarrow}0$ since $\hat{\gamma}_n\overset{\mathbb{P}}{\rightarrow}\gamma.$
\section{Appendix}

\begin{lemma}\label{dony_einmahl}
For fixed $n \geq 1 $, let $(Y_i)_{1 \leq i \leq n}$ be a sequence of i.i.d. random variables taking values in $(\mathfrak{X},\mathcal{X})$.
Let be $(\epsilon_i)_{1 \leq i \leq n}$ a sequence of independent Bernoulli random variables independent from $Y_i$.
Write $\nu (k):= \underset{i=1}{\overset{k}{\sum}}\epsilon_i$ for $k \leq n$. We have
\begin{equation}\label{eq:distribution_permutation_sum}
\underset{i=1}{\overset{n}{\sum}}\delta_{Y_i}\epsilon_i\overset{\mathscr{L}}{=}\underset{i=1}{\overset{\nu (n)}{\sum}}\delta_{Y_i},
\end{equation}
where the equality in law is understood as on the sigma algebra spanned by all Borel positive functions on $(\mathfrak{X},\mathcal{X})$.
Moreover, for every $g$ positive and measurable function, we have
\begin{equation}\label{eq:distribution_permutation_max}
\underset{i=1...n}{\max}g(Y_i)\epsilon_i\overset{\mathscr{L}}{=}\underset{i=1...\nu(n)}{\max}g(Y_i).
\end{equation}
\end{lemma}

\begin{proof}
Let $e \in \lbrace0;1\rbrace^n$, and let $g$ be real measurable and positive function. Since the variables $(Y_i)_{1 \leq i \leq n}$ are i.i.d and independent from $(\epsilon_i)_{1 \leq i \leq n}$ we have, for any given permutation $\sigma$ of $\llbracket 1,n \rrbracket$, $$\Big(g(Y_1),...,g(Y_n)\Big)\overset{\mathscr{L}}{=}\Big(g(Y_1),...,g(Y_n) |\epsilon = e\Big)\overset{\mathscr{L}}{=} \Big(g(Y_{\sigma(1)}),...,g(Y_{\sigma(n)}) |\epsilon = e \Big)$$ by exchangeability. Now define $\sigma$ by 

$$
\sigma(k) = \left\{
    \begin{array}{ll}
        \underset{j=1}{\overset{i}{\sum}}e_j \text{ if } e_i=1 \\
        n-\underset{j=1}{\overset{i}{\sum}}1-e_j  \text{ if } e_i=0
    \end{array}
\right. 1\leq i \leq n .
$$
Write $s(e)=\underset{i=1}{\overset{n}{\sum}}s(e_i)$ for the total number of ones in $(e_1,...,e_n)$. By construction, the indices $i$ for which $e_i=1$ are mapped injectively to the set of first indices $\llbracket 1,s(e)\rrbracket$, while those for which $e_i=0$ are injectively mapped into $\llbracket s(e)+1,n\rrbracket$. \\
As $e$ has fixed and non random coordinates, we have
$$\Big(g(Y_1)e_1,...,g(Y_n)e_n \Big|\epsilon = e\Big)\overset{\mathscr{L}}{=} \Big(g(Y_{\sigma(1)})e_1,...,g(Y_{\sigma(n)})e_n \Big|\epsilon = e\Big) .$$
Hence
\[
\begin{array}{cll}
 \left. \sum_{i=1}^n g(Y_i) e_i \Big| \epsilon=e \right.&
\overset{\mathscr{L}}{=}\sum_{i=1}^n g(Y_i) e_i & \\
&\overset{\mathscr{L}}{=}\sum_{i=1}^n g(Y_{\sigma(i)}) e_{\sigma(i)}& \\
&\overset{\mathscr{L}}{=}\sum_{i=1}^{s(e)} g(Y_{\sigma(i)}) \quad \text{ (a)}&\\
&\overset{\mathscr{L}}{=}\sum_{i=1}^{s(e)} g(Y_{i}) \quad \quad \text{ (b)}&\\
&\overset{\mathscr{L}}{=}\left. \sum_{i=1}^{s(e)} g(Y_i) e_i \Big| \epsilon=e \right. &
\overset{\mathscr{L}}{=}\left. \sum_{i=1}^{s(\epsilon)} g(Y_i) e_i \Big| \epsilon=e \right. ,
\end{array}
\]
where (a) is because $e_{\sigma(i)}=0$ for $i>s(e)$ by construction and (b) is obtained by noticing that $F_e(y_{\sigma(1)},...,y_{\sigma(n)})\overset{\mathscr{L}}{=}F_e(y_1,...,y_n)$ with 
\[
F_e:(y_1,...,y_n) \mapsto \sum_{i=1}^{s(e)}g(y_i).
\]
Unconditioning upon $\epsilon$ gives (\ref{eq:distribution_permutation_sum}).
The same proof with $F_e(y_1,...,y_n)=\max\{g(y_i), 1\leq i \leq s(e)\}$ gives the second equality (\ref{eq:distribution_permutation_max}).

\end{proof}

\bibliographystyle{plain}
\bibliography{ref_these}

\end{document}